\documentclass[bj]{imsart}

\RequirePackage{amsthm,amsmath,amsfonts,amssymb}
\RequirePackage[numbers]{natbib}

\usepackage{bbm}
\usepackage[fixamsmath]{mathtools}
\usepackage[ruled,linesnumbered]{algorithm2e}
\usepackage{graphicx}
\usepackage{caption}
\usepackage{hyperref}
\usepackage{cleveref}
\crefname{assumption}{assumption}{assumptions}
\usepackage{nicefrac}
\usepackage{enumitem}

\usepackage{todonotes}

\startlocaldefs
\numberwithin{equation}{section}
\theoremstyle{plain}
\newtheorem{theorem}{Theorem}
\newtheorem{lemma}{Lemma}
\newtheorem{corollary}{Corollary}
\newtheorem{proposition}{Proposition}
\newtheorem*{claim*}{Claim}
\theoremstyle{remark}
\newtheorem{assumption}{Assumption}
\newtheorem{example}{Example}
\newtheorem{definition}{Definition}
\newtheorem{remark}{Remark}
\newcommand{\E}{\mathbb{E}}
\newcommand{\N}{\mathbb{N}}
\renewcommand{\P}{\mathbb{P}}
\newcommand{\R}{\mathbb{R}}
\newcommand{\cA}{\mathcal{A}}
\newcommand{\cB}{\mathcal{B}}

\newcommand{\cE}{\mathcal{E}}
\newcommand{\cF}{\mathcal{F}}

\newcommand{\cN}{\mathcal{N}}
\newcommand{\cO}{\mathcal{O}}

\newcommand{\cS}{\mathcal{S}}
\newcommand{\cV}{\mathcal{V}}

\newcommand{\Ew}[1]{\mathbb{E}\left[#1\right]}
\renewcommand{\Pr}[1]{\mathbb P\left(#1\right)}
\newcommand{\norm}[1]{\lVert #1 \lVert}
\newcommand{\Ind}[1]{\mathbbm{1}_{ #1 }}

\endlocaldefs

\begin{document}

\begin{frontmatter}
\title{Distributed gradient-based optimization in the presence of dependent aperiodic communication}

\begin{aug}
\author[A]{\inits{A.}\fnms{Adrian}
\snm{Redder$^{*}$}\ead[label=e1]{aredder@mail.upb.de}},
\author[B]{\inits{A.}\fnms{Arunselvan}
\snm{Ramaswamy}\ead[label=e2]{arunselvan.ramaswamy@kau.se}}
\and
\author[C]{\inits{H.}\fnms{Holger} \snm{Karl}\ead[label=e3]{holger.karl@hpi.de}}
\address[A]{Department of Computer Science, Paderborn University. \printead{e1}}
\address[B]{Department of Computer Science, Karlstad University. \printead{e2}}
\address[C]{Hasso-Plattner-Institute, University of Potsdam. \printead{e3}}
\end{aug}

\begin{abstract}
	Iterative distributed optimization algorithms involve multiple agents that communicate with each other, over time, in order to minimize/maximize a global objective. In the presence of unreliable communication networks, the Age-of-Information (AoI), which measures the freshness of data received, may be large and hence hinder algorithmic convergence. In this paper, we study the convergence of general distributed gradient-based optimization algorithms in the presence of communication that neither happens periodically nor at stochastically independent points in time. We show that convergence is guaranteed provided the random variables associated with the AoI processes are stochastically dominated by a random variable with finite first moment. This improves on previous requirements of boundedness of more than the first moment. We then introduce stochastically strongly connected (SSC) networks, a new stochastic form of strong connectedness for time-varying networks. We show: If for any $p \ge0$ the processes that describe the success of communication between agents in a SSC network are $\alpha$-mixing with $n^{p-1}\alpha(n)$ summable, then the associated AoI processes are stochastically dominated by a random variable with finite $p$-th moment. In combination with our first contribution, this implies that distributed stochastic gradient descend converges in the presence of AoI, if $\alpha(n)$ is summable.

\end{abstract}

\begin{keyword}
\kwd{Distributed Optimization}
\kwd{Stochastic Gradient Decent}
\kwd{Time-Varying Networks}
\kwd{Age of Information}
\kwd{$\alpha$-Mixing}
\kwd{Stochastic dominance}

\end{keyword}

\end{frontmatter}


\section{Introduction}
\label{sec: intro}

Distributed optimization of stochastically approximated loss functions lies at the heart of many system-level problems that arise in multi-agent learning \cite{redder2022asymptotic}, resource allocation for data centers \cite{haghshenas2019magnetic}, or decentralized control of power systems \cite{lin2017decentralized}. In these scenarios, distributed implementations have many advantages such as balanced workload or the avoidance of a single point of failure. However, this usually comes with high communication costs for coordination \cite{tang2018communication}, entailing that information can only be exchanged rarely, causing local versions of global information to be significantly outdated. Hence, it is of high interest to characterize conditions such that a distributed optimization algorithm can converge when only significantly outdated information with sporadic updates is available.

We therefore consider distributed stochastic optimization problems (SOPs) where the choice of local optimization variables has to be coordinated over an uncertain time-varying communication network. A typical distributed SOP can take the following form:
\begin{equation}
\label{eq: intro_problem}
\begin{split}
x^* = (x_1^*, \ldots, x_D^*) &= \underset{x \in \R^d}{\text{argmin }} \E_\xi \left[ f(x,\xi)\right]
\end{split}
\end{equation}
The objective is to minimize a real-valued function $f:\R^d \times \cS \to \R$, which is function of an optimization variable $x \in \R^d$ and a random variable $\xi$ representing noise or uncertainty taken from a set $\cS$. The optimization variable $x$ is composed of local components $x_i \in \R^{d_i}$ that are associated with local agents of a distributed system. Hence, no global control of $x$ is possible. Moreover, the distribution of $\xi$ is typically unknown in practical scenarios and the agents can only observe samples $\xi^n$ of the uncertainty $\xi$ at discrete time steps $n \in \N_0$. Thus, the problem is that each agent~$i$ has to coordinate the local choice for its variable $x_i$ with all other agents by exchanging information over a network and iteratively refine the choice based on the observed samples of uncertainty $\xi^n$.


To solve this problem, we propose the following solution. Suppose every agent~$i$ runs a local distributed stochastic gradient descent (SGD) algorithm that generates a sequence $\{x_i^n\}_{n=0}^\infty$ to solve problem \eqref{eq: intro_problem}. Ideally, every agent $i$ would like to have direct access to every new element of the sequences $\{x_j^n\}_{n=0}^\infty$ from every other agent $j\not= i$ during run-time of its own local algorithm. However, due to the distributed nature of the considered setting, the agents have to communicate their updates for their local optimization variables to other agents via a communication network.
Because of the uncertainty of communication networks, each agent~$i$ can therefore only use delayed versions $x_j^{n-\tau_{ij}(n)}$ for all $j \not= i$ to update its own local variable $x_i^n$. Here, $x_j^{n-\tau_{ij}(n)}$ denotes the newest update of $x_j$ available at agent~$i$ at time $n$ and $\tau_{ij}(t)$ is its corresponding age. \emph{We refer to the $\tau_{ij}(t)$'s as the Age of Information (AoI) variables.}
The resulting distributed algorithm is therefore in essence a ``straightforward'' implementation of SGD, where the true values of local variables are replaced by their aged counterparts. Due to the size of generated information in large distributed systems, and the uncertainty and high cost for communication over networks, the AoI variables can not be expected to be bounded and should therefore be modelled as an unbounded sequence of random variables.
\emph{The problem is therefore to formulate mild network and communication assumptions that are representative and easily verifiable such that this SGD algorithm that uses highly aged information will still converge.}

A major challenge for this problem is the multitude of potential factors that affect the AoI random variables. Information exchange between some pairs of agents might experience unbounded delays; mobility of agents or network scheduling algorithms can induce a varying set of network topologies. This can create dependencies among successive network transmissions, preventing agents from exchanging data for extended periods of time. In general, transmissions that happen close in some domain (e.g. time,  frequency, or space in wireless communication) are expected to be highly correlated. It is therefore important to formulate a communication network model and associated assumptions that can represent these cases while being mathematically tractable for analysis. Notably, the assumption of guaranteed periodic or stochastically independent communication is practically unrealistic.

\subsection{Network models in the literature}

One of the most common models in the distributed optimization literature is a time-varying network model that is represented by a time-varying graph (\Cref{def:tv_network}). For this graph, the most common assumption is that there is a constant $M$ such that the union graph associated with all time intervals $[n, n+M]$ is strongly connected \cite{wang2019distributed, aybat2019distributed, Yu2020-gq,kovalev2021adom}. A network with this property is typically called uniformly strongly connected \cite{nedic2014distributed}, $M$-strongly connected \cite{nedic2017achieving,Scutari2019-sa} or jointly strongly connected \cite{xu2017distributed}. This model implies guaranteed periodic communication. Another common model is to assume a time-varying network graph whose expected union graph is strongly connected, where the events that describe the success of communication across network edges are independent across time \cite{bastianello2020asynchronous, Lei2018-ag, koloskova2020unified, ram2022_tac}. 

In ref.~\cite{Lei2018-ag, nedic2017achieving, wang2019distributed, Scutari2019-sa, Yu2020-gq, koloskova2020unified, bastianello2020asynchronous, kovalev2021adom} 
the objective is that agents come to a consensus on one global optimization variable to minimize the sum of real-valued functions, each of which associated with one of the local agents. Although such consensus-type problems might appear quite different to \eqref{eq: intro_problem}, it turns out that an algorithm for \eqref{eq: intro_problem} can also find a solution for consensus problems after a minor reformulation at the cost of additional communication, which we discuss in \cite{redder2022constraint}.  In contrast to the consensus type problems, ref.~\cite{xu2017distributed, aybat2019distributed, ram2022_tac} and this work consider distributed optimization problems where each agent has to select a local optimization variable, such that the combination of all local variables solves a global optimization problem. 

Observe that literature exclusively considers network models that either guarantee periodic communication or require communication based on independent events. We believe that these are restrictive assumptions that do not represent real-world communication networks well. To close this gap, we present a less restrictive network model and verifiable network conditions that guarantee that an SGD algorithm finds a solution to problems of the form \eqref{eq: intro_problem}. We also show the aforementioned typical network assumptions from the literature are stronger versions of our new set of network assumptions (\Cref{asm: network-A,asm: network-B}). Our assumptions only require a stochastic form of strong connectivity and a dependency decay (mixing) property.  \emph{To the best of our knowledge, ours is the first work that guarantees asymptotic convergence of a distributed optimization scheme under such mildly restrictive conditions, connecting an abstract optimization theory with a wide range verifiable network conditions.} However, it must be noted that other papers (such as those discussed above) can provide rate of convergence results while we merely give an almost sure convergence analysis. 



\subsection{Summary of contribution}
Our work contributes to the literature of network conditions that guarantee asymptotic convergence to the set of stationary points of a distributed stochastic optimization problems with potentially non-convex objective function.
Our work, builds on our previous work on SGD for time-varying networks \cite{ram2022_tac}. However, whereas in \cite{ram2022_tac} the focus was on the optimization iteration, with a strong and restrictive i.i.d. network assumption, this work focuses on guaranteeing convergence under significantly weaker network conditions. Most importantly, our network conditions cover time-varying network topologies, unbounded communication delays, non-independent aperiodic communication, asynchronous local updates and event-driven communication. 

As the first step, we describe a distributed stochastic gradient descend algorithm (\Cref{alg: 1}) that instead of true local variables uses aged variables as a consequence of network communication. The AoI variables $\tau_{ij}(n)$ therefore induce gradient errors when comparing \Cref{alg: 1} with and without AoI. 
As our first major contribution, we show in \Cref{lem:vanishing_errors} that the aforementioned gradient errors vanish asymptotically under an asymptotic growth conditions for the AoI variables. Specifically, we require that all $\tau_{ij}(n)$ for all $n \in \N_0$ are stochastically dominated by a non-negative integer-valued random variable with \emph{at least finite first moment}. This provides a significant weakening of traditional assumptions from the stochastic approximation literature in the present setting, since traditionally a dominating random variable with at least a bounded moment greater than one was required. With \Cref{lem:vanishing_errors} we then show the convergence of \Cref{alg: 1} in \Cref{thm:convergence}.

Our second contribution is a universally applicable time-varying network model and associated assumptions to generally verify the existence of  dominating random variables with an arbitrary required moment conditions. Our time-varying network model is formulated using events $A^n_{ij}$, each of which represents successful information exchange from some agent~$i$ to another agent~$j$ during some time slot $n$. We then introduce the notion of $(\varepsilon,\kappa)$-stochastically strongly connected (SSC) network with $\varepsilon \in (0,1)$ and $\kappa \in \N_0$. This notion requires that there is a set of network edges that form a strongly connected graph for which $\Pr{\cup_n^{n+\kappa} A_{ij}^n)} \ge \varepsilon$ for all $n \in N_0$. In other words, for those edges communication is successful at least ones over every interval of length $\kappa$ with at least probability $\varepsilon$. We then present a general recipe to validate stochastic dominance properties with required moment conditions. Afterwards, \Cref{thm: 0} presents our main result: Fix any $p\ge0$ and consider a $(\varepsilon,\kappa)$-SSC network. If there exists some $\eta \in \N_0$, such that the processes $\Ind{\bigcup_{k=n}^{n+\eta} A^n_{ij}}$ are $\alpha$-mixing with $\sum_{n=0}^{\infty}n^{p-1} \alpha(n)  < \infty$, then all AoI variables $\tau_{ij}(n)$ are stochastically dominated by non-negative integer-valued random variable $\overline{\tau}$ with $\Ew{\overline{\tau}^p} < \infty$. This result, together with \Cref{thm:convergence}, will therefore imply our final convergence result for \Cref{alg: 1} under the minimal requirement of a SSC network with summable $\alpha$-mixing coefficients.

The rest of the paper is structured as follows:
In \Cref{sec:notation} we state notation and preliminaries from probability and graph theory. In \Cref{sec: problem} we discuss the problem formulation and our distributed SGD algorithm. Afterwards we prove the almost sure convergence of \Cref{alg: 1} in \Cref{sec: cov_ana} under asymptotic growth conditions for the AoI variables. We then introduce our time-varying network model and associated assumptions in \Cref{sec:network_conditions}. \Cref{sec: analysis} then presents our construction to validate stochastic dominance properties and our main results.
Finally, we discuss the verifiability of our network assumptions and future work in \Cref{sec:conclusion}.



\section{Notation, definitions and preliminaries}
\label{sec:notation}
This section presents notation, and preliminaries from probability and graph theory. Throughout our work, discrete points in time are indicated by superscript letters $n$. We refer to a time slot $n$ as the time interval from time step $n-1$ to $n$. We use $n \in \N_0$ to denote $n \in \N\cup \{0\}$. 

We make frequent use of the big $\cO$ notation: Consider two real-valued sequences $x^n$, $y^n$. Then  $x^n \in \cO(y^n)$, if $\limsup\limits_{n \rightarrow \infty} \frac{x^n}{y^n}  < \infty$.

From probability theory we need the concepts of stochastic dominance, expectation of non-negative integer-valued random variables, measure of dependency and $\alpha$-mixing:

\begin{definition}
	\label{def: stochastic_dominance}
	A non-negative integer-valued random variable $\tau$ is said to be \textbf{stochastically dominated} by a random variable $\overline{\tau}$ if $\Pr{\tau > m } \le \Pr{\overline{\tau} > m}$ for all $m \ge 0$.
\end{definition}

\begin{proposition}[\cite{chakraborti2018higher}]
	\label{eq:moment_eq}
	Suppose $\tau$ is non-negative integer-valued random variable, then
	\begin{equation}
	\Ew{\tau^p} = \sum_{m=0}^{\infty} ((m+1)^p-m^p)\Pr{\tau > m}, \qquad p >0.
	\end{equation}
\end{proposition}

Let $(\Omega, \cF, \P)$ be a probability space, and let $\cA$ and $\cB$ be two sub-$\sigma$-algebras of $\cF$. 
The measure of dependency $\alpha$ between $\cA$ and $\cB$ is defined as
	\begin{equation}
	\alpha(\cA,\cB) \coloneqq \sup_{A \in \cA, B \in \cB} \lvert \Pr{A \cap B} - \Pr{A}\Pr{B} \rvert.
	\end{equation}
Consider a stochastic process $X = \{X_n\}_{n \in \N_0}$. For $0 \le l \le m \le \infty$, define the sub-$\sigma$-algebra generated from $X_l$ up to $X_m$ by
\begin{equation}
\cF_l^m \coloneqq \sigma\left(X_n \mid l \le n \le m\right),
\end{equation}
Informally, the $\sigma$-algebra generated by a stochastic process from a time interval describes the information that can be extracted from the associated process realizations, see \cite{durrett2019probability} for details.
\begin{definition}
	\label{def:mixing}
	The $\alpha$-mixing coefficients of the process $X$ are
	\begin{equation}
	\label{eq: alpha_mix}
	\alpha(n) \coloneqq \sup_{l\ge 0} \alpha(\cF_0^l, \cF_{l+n}^\infty).
	\end{equation}
	for every $n \in \N_0$. The process $X$ is called \textbf{strongly mixing} (or $\alpha$-mixing), if $\alpha(n) \rightarrow 0$ as $n \rightarrow \infty$.
\end{definition}

Mixing is a notion of asymptotic independence. We refer to \cite{bradley2005basic} for a survey about different mixing notions. We now introduce a subclass of strongly mixing processes with different rates of convergence:
\begin{definition}
	\label{def: mixing_process}
	The process $X$ is called $\boldsymbol{p}$\textbf{-strongly mixing} for some $p \ge 0$, if
	\begin{equation}
	\sum_{n=0}^{\infty} n^{p-1} \alpha(n) < \infty.
	\end{equation}
\end{definition}
We will use this new mixing property to describe dependency decay of different orders. 

For details on graph theory we refer the reader to \cite{trudeau1993introduction}. We require the following concepts: 

\begin{definition}
	A directed graph is called \textbf{strongly connected}, if every pair of nodes is connected by a directed path.
\end{definition}

\begin{definition}
	\label{def:tv_network}
	A \textbf{time-varying network} is defined as sequence
	\begin{equation}
	\{(\cV,\cE^n)\}_{n \in \N_0},
	\end{equation}
	where each element $(\cV,\cE^n)$ is a directed graph.
\end{definition}

We will use the following new connectivity notion for time-varying networks:
\begin{definition}
	\label{def:SSC}
	A time-varying network $\{(\cV,\cE^n)\}_{n \in \N_0}$ is called \textbf{$(\varepsilon, \kappa)$-stochastically strongly connected (SSC)} with $\varepsilon \in (0,1)$ and $\kappa \in \N_0$ , if there exists a strongly connected graph $(\cV, \cE)$, such that for all  $n \in \N_0$ and
	for all $(i,j) \in \cE$
	\begin{equation}
	\Pr{(i,j) \in \bigcup_{k=n}^{n+\kappa} \cE^k} \ge \varepsilon.
	\end{equation}
\end{definition}


\section{Problem description}
\label{sec: problem}

We consider a $D$-agent distributed optimization problem, where each agent $i \in \cV \coloneqq \{1,\ldots,D\}$ has to choose values for a local variable $x_i \in \R^{d_i}$ to minimize a global objective function $F$. The global optimization variable $x = (x_1, \ldots, x_D) \in \R^d $ is the concatenation of the local optimization variables $x_i$ associated with the local agents. The objective function is assumed to be stochastic and given by
\begin{equation}
\label{eq: unconstrained_problem}
F(x) \coloneqq \E_\xi \left[ f(x,\xi)\right],
\end{equation}
with $f:\R^d \times \cS \to \R$ a random real-valued function, where the randomness is modeled by an $\cS$-valued random variable $\xi$ that represents noise or uncertainty. 

As discussed in the introduction, if a central agent had direct control of the optimization vector $x$, it would be straightforward to find a local minimum of \eqref{eq: unconstrained_problem} using stochastic gradient descend (SGD) under suitable assumptions \cite[Ch. 10]{borkar2009stochastic}. However, as the components $x_i$ of $x$ are associated with distributed agents, we consider that the agents need to coordinate their choice for the local optimization variables by exchanging information via a communication network.

We assume a synchronized communication setting according to a global clock $n \in \N_0$. Each agent updates its local variable at every time step $n$ based on a local gradient descend iteration. The iterations will be defined in \Cref{sec: iteration and algorithm}. For each agent $i\in \cV$, the local iterations generate a sequence $\{x_i^n\}^\infty_{n= 0}$ starting from an initial candidate value $x^0_i$ for the optimal value $x^*_i$. 
To execute the local gradient iteration, agent~$i$ requires a locally available estimate of the current optimization variable $x^n_j$ of agent~$j$ for all $j\not=i$. We consider that these information have to be communicated using a communication network. Specifically, every agent will use the newest available local optimization variable from every other agent to update its own local variable. Due to the potential uncertainty of the network only aged/delayed versions of the local variables of the other agents are available at agent~$i$ at any time step. 
Therefore, agent~$i$ has only access to the delayed version
\begin{equation}
\label{eq:believe_vector}
\hat{X}_i^n \coloneqq \left(x_1^{n-\tau_{i1}(n)}, \ldots, x_i^n, \ldots, x_D^{n-\tau_{iD}(n)} \right)
\end{equation}
of $x^n$ at every time step $n$. Here, $x_j^{n-\tau_{ij}(n)}$ denotes the newest update of $x_j$ available at agent~$i$ at time $n$ and $\tau_{ij}(t)$'s are the AoI random variables. Further, we refer to $\hat{X}_i^n$ as the \emph{local believe vector} of agent~$i$ at time $n$. As the next step, we describe the gradient based iteration that uses $\hat{X}_i^n$ instead of $x^n$ to solve problem \eqref{eq: unconstrained_problem}. 

\subsection{Algorithm}
\label{sec: iteration and algorithm}

We consider that the agents iteratively refine their local variables using the partial derivatives $\nabla_{x_j} f(\cdot, \xi)$.  We assume that the agents do not know the distribution of $\xi$, but during any time slot $n$ an agent can observe an i.i.d.\ realisation $\xi^n$ of $\xi$. For simplicity, we assume that all agents are affected by the same realisation of the random variable $\xi$. In other words, when agent~$i$ and agent~$j$ calculate their partial derivatives during some time slot $n$, they use the same realisation $\xi^n$ of $\xi$, i.e. $\nabla_{x_i} f(\cdotp, \xi^n)$ and $\nabla_{x_j} f(\cdotp, \xi^n)$. The extension to agent-specific realisations of $\xi$ is merely a technical reformulation that was already described in \cite{arxiv}.

To evaluate the partial derivatives $\nabla_{x_i} f(\cdotp, \xi^n)$, agent~$i$ uses the most recent available version of the optimization variable $x^n_j$ of agent~$j$ for all $j\not=i$, i.e. it calculates $\nabla_{x_i} f(\hat{X}^n_i, \xi^n)$ instead of $\nabla{x_i} f(x^n, \xi^n)$. The following SGD iteration is used by each agent to update its local variable:
\begin{equation}
\label{eq: real_iteration}
x^{n+1}_i =  x^n_i -a(n)\left( \nabla_{x_i} f(\hat{X}^n_i, \xi^n) + \lambda_i^n \right), 
\end{equation}
where $\{a(n)\}_{n \in \N_0}$ is a given step-size sequence and $\lambda_i^n$ is a local stochastic additive error that may arise during the calculation of $\nabla_{x_i}f$. \Cref{alg: 1} summarizes the protocol that runs on every agent locally. For now, we assume that the agents use some communication protocol to exchange their local believe vectors $\hat{X}^n_i$ over a network. The protocol and the network properties therefore induce the distribution of the AoI variables $\tau_{ij}(n)$. In the next section, we will proof the convergence of \Cref{alg: 1} under an abstract growth conditions for the AoI variables. 
\Cref{sec:network_conditions} then formulates a communication network model and associated assumptions to satisfy these growth conditions.  

\begin{remark}
	In our previous work \cite{ram2022_tac}, we also included asynchronous gradient updates in \eqref{eq: real_iteration}. The agents are therefore allowed to update their local variable not at every time step $n\ge 0$. This may be included here using the associated assumptions from \cite{ram2022_tac}. Our previous work, considers \eqref{eq: real_iteration} for a restrictive network model with independent communication (see \Cref{sec: comparison} for further details). This work resolves this issue, but we use synchronous gradient updates to avoid notational overload.  
\end{remark}

\begin{algorithm}[!t] 
	\SetAlgoLined
	Initialize local optimization variable estimate $x^0_i$ \; 
	Initialize local belief vector $\hat{X}^0_i$ \;
	\For{all time steps $n$}
	{
		Obtain network sample $\xi^n$ \;
		$x^{n+1}_i \leftarrow x^n_i - a(n) \nabla_{x_i} f(\hat{X}^n_i, \xi^n)$ \;
		Update $i$-th component of $\hat{X}^{n}_i$ to new $x^{n+1}_i$ with appended timestamp $n+1$ \;
		Run local communciation protocol to exchange $\hat{X}^{n}_i$ \;
	}
	\caption{Local algorithm at agent $i \in \cV$}
	\label{alg: 1}
\end{algorithm}

\section{Asymptotic convergence of \Cref{alg: 1}}
\label{sec: cov_ana}

In this section, we will show the asymptotic convergence of \Cref{alg: 1}. Specifically, we show that the iterations in \eqref{eq: real_iteration} converge to a neighbourhood of a local stationary point of \eqref{eq: unconstrained_problem}. The main part of the proof is to show that the gradient errors 
\begin{equation}
\label{eq:grad_error}
e^n_i \coloneqq  \nabla_x F(x_1 ^{n - \tau_{i1}(n)}, \ldots, x_D ^{n - \tau_{iD}(n)}, \xi^n) - \nabla_x F(x_1 ^n, \ldots, x_D ^n, \xi^n)
\end{equation}
due to AoI vanish asymptotically. This error captures the difference of the gradient descent step some agent~$i$ would take given its local believe vector compared to the true global state.

To show that the gradient errors vanish, we require that the AoI variables $\tau_{ij}(n)$ satisfy an asymptotic growth conditions. Observe that the gradient errors depend on the AoI variables and the step size sequence $a(n)$, since $a(n)$ determines how much successive steps of iteration \eqref{eq: real_iteration} differ. If the step size sequence gets smaller quick enough relative to some maximal potential growth of the AoI variables, we expect $e^n_i$ to decay to zero. This is because even significantly outdated information stays relevant, if the steps taken during that time were comparably small. The convergence of \Cref{alg: 1} will then follow from the convergence of \eqref{eq: real_iteration} when one considers no AoI, i.e. the case $\tau_{ij}(n) = 0$. 

The following assumption formalize the required trade of between the choice of the step size sequence and the required network quality. 
\begin{assumption}
	\label{asm: step_size}
	\begin{enumerate}
		\item  There exists $p \in [1, 2)$ and a non-negative integer-valued random variable $\overline{\tau}$, such that $\overline{\tau}$ stochastically dominates (\Cref{def: stochastic_dominance}) all $\tau_{ij}(n)$ for all $i,j \in \cV$ and all $n \in \N_0$ with 
		$$ \Ew{\overline{\tau}^p} < \infty.$$		
		\item The step-size sequence $\{a(n)\}_{n \in \N_0}$ satisfies:
		\begin{enumerate}
			\item[(i)] $\sum \limits_{n = 0}^\infty a(n) = \infty$, $\sum \limits_{n = 0}^\infty a(n) ^2 < \infty$.
			\item[(ii)] $a(n) \in \cO(n^{-\frac{1}{p}})$ with $p$ as in 1.
		\end{enumerate} 
	\end{enumerate}
\end{assumption}

\Cref{asm: step_size}.1 requires that the network quality is good enough, such that the tail distribution of the AoI variables $\tau_{ij}(n)$ decays rapidly enough, such that at least a dominating random variable with finite mean exists. This assumption contributes a significant weakening of the traditional assumptions required for convergence in the present setting. The traditional assumptions formulated in \cite{borkar1998asynchronous}, required at least a dominating random variable with finite $p$-th moment for $p>1$. In this work, we show for the first time that actually $p=1$ is sufficient to achieve asymptotic convergence. 
We show that under \Cref{asm: step_size}.1 the growth of each $\tau_{ij}(n)$ can not exceed any fraction of $n$ after some potentially large time-step. We formulate this in \Cref{lem:finite_CDFsum}.

\Cref{asm: step_size}.1(i) is standard in the stochastic approximation literature. \Cref{asm: step_size}.1(ii) requires that we choose the step size depending on the network quality. For example, if only the worst network quality can be verified, i.e. that there is only a dominating variable with finite mean, then we have to choose $a(n) \in \cO(\frac{1}{n})$. In addition to the aforementioned weakening of assumptions, we also do not require that the step-size sequence is eventually monotonically decreasing and we only require $a(n) \in \cO(n^{-\frac{1}{p}})$ instead of $a(n) \in o(n^{-\frac{1}{p}})$. Both conditions were traditionally assumed.

We will now present additional assumptions associated with the objective function $f$ in \eqref{eq: unconstrained_problem} and the iterations in \eqref{eq: real_iteration}. After that we show the convergence of \Cref{alg: 1}. In \Cref{sec:network_conditions} we will then present verifiable network conditions to ensure that \Cref{asm: step_size}.1 holds. We will also see that it is easy to formulate very restrictive network conditions, such that the growth of the AoI variables behave very well. For example, one can show that all moments of the AoI variables are bounded under the standard assumptions in the distributed optimization literature (see \Cref{sec: comparison}). That is, \Cref{asm: step_size}.1 would be satisfied for all $p\ge1$.

In addition to \Cref{asm: step_size}, we require the following assumptions.
\begin{assumption}
	\label{asm: objective}
	\begin{enumerate}
		\item $\nabla_x f$ is continuous and locally Lipschitz continuous in the $x$-coordinate, where the associated constant may depend on $\xi$. 
		\item $\Ew{\nabla_{x}f} = \nabla_{x} \Ew{f}$.
		\item $\xi$ is an $\mathcal{S}$-valued random variable, where $\mathcal{S}$ is a one-point compactificable space.
	\end{enumerate}
\end{assumption}
\begin{assumption}
	\label{asm: stability}
	For all $i\in \cV$, we have $\sup \limits_{n \in \N_0} \ \lVert x_i^n \rVert < \infty$ a.s.
\end{assumption}
\begin{assumption}
	\label{asm: additive_error}
	Almost surely, $\limsup\limits_{n \to \infty} \norm{\lambda^n} \le \lambda$ for some fixed $\lambda > 0$.
\end{assumption}
We refer to \cite{ram2022_tac} for detailed discussion on the verifiability of \Cref{asm: objective,asm: stability,asm: additive_error}.

Recall the gradient errors due to the AoI variables in \eqref{eq:grad_error}. Next, we will show that these gradient errors vanish asymptotically. We start with an asymptotic grow property for the AoI variables under \Cref{asm: step_size}.1. 

\begin{lemma}
	\label{lem:finite_CDFsum}
	Under \Cref{asm: step_size}.1, we have for every $\varepsilon \in (0,1)$ and for all $i,j \in \cV$ that
	\begin{equation}
	\sum_{n=0}^{\infty} \Pr{\tau_{ij}(n) > \varepsilon n^{\frac{1}{p}}} < \infty.
	\end{equation}
\end{lemma}
\begin{proof}
	Fix $\varepsilon \in (0,1)$. By \Cref{asm: step_size} there is a non-negative integer-valued random variable $\overline{\tau}$, such that
	\begin{equation}
	\Pr{\tau_{ij}(n) > \varepsilon n^{\frac{1}{p}}} \le \Pr{\overline{\tau}> \varepsilon n^{\frac{1}{p}}}
	\end{equation}
	for all $n \in \N_0$ and $\Ew{\overline{\tau}^p} < \infty$.
	Hence, we have
	\begin{align}
	\sum_{n=0}^{\infty} \Pr{\tau_{ij}(n) > \varepsilon n^{\frac{1}{p}}} &\le \sum_{n=0}^{\infty} \Pr{\overline{\tau} > \varepsilon n^{\frac{1}{p}}}\\ &= \sum_{m=0}^{\infty} \sum_{n \in \cN(m)}\Pr{\overline{\tau} > \varepsilon n^{\frac{1}{p}}} 
	\\ &\le \sum_{m=0}^{\infty} \sum_{n \in \cN(m)}\Pr{\overline{\tau} > m} = \sum_{m=0}^{\infty} \lvert \cN(m) \rvert \Pr{\overline{\tau} > m},
	\end{align}
	where the sets $\cN(m)$ are defined as
	\begin{equation}
	\cN(m) \coloneqq \{n \in \N_0 : m \le \varepsilon n^{\frac{1}{p}} < m+1\} 
	= \{n \in \N_0 : \nicefrac{m^p}{\varepsilon^p} \le n < \nicefrac{(m+1)^p}{\varepsilon^p}\}
	\end{equation}
	for every $m \in \N_0$. We use these sets to consider all $\varepsilon n^{\frac{1}{p}}$ in every interval $[m,m+1)$. The second inequality then follows from the monotonicity of the cumulative distribution function (CDF) by definition of the sets $\cN(m)$.
	Since $\lvert \cN(m) \rvert \le \frac{1}{\varepsilon^p}\left((m+1)^p - m^p\right)$, we have therefore shown that
	\begin{equation}
	\label{eq:sumCDF_ineq}
	\sum_{n=0}^{\infty} \Pr{\tau_{ij}(n) > \varepsilon n^{\frac{1}{p}}} \le \frac{1}{\varepsilon^p}\sum_{n=0}^{\infty}\left((n+1)^p - n^p\right) \Pr{\overline{\tau}   >  n} = \frac{1}{\varepsilon^p} \Ew{\overline{\tau}^p} < \infty.
	\end{equation}
	The last equality follows from \Cref{eq:moment_eq}, since $\overline{\tau}$ is a non-negative integer-valued random variable.
\end{proof}
We are now ready to prove that the gradient errors due to AoI vanish asymptotically.
\begin{lemma}
	\label{lem:vanishing_errors}
	Under \Cref{asm: objective,asm: stability,asm: step_size}, we have that $\lim\limits_{n \to \infty} \norm{e^n_i} = 0$.
\end{lemma}
\begin{proof}
	By \Cref{asm: stability}, we have that $x^n \in \overline{B}_r(0)$ for some sample path dependent radius $0<R<\infty$. Then, \cite[Lemma 1]{ram2022_tac} shows that $\nabla_{x}F$ is locally Lipschitz continuous with a constant independent of $\xi$.
	Hence, $\nabla_{x}F$ is globally Lipschitz continuous with a constant $L$ when restricted to $\overline{B}_R(0)$. Using the triangular inequality, the established Lipschitz continuity of $\nabla_{x}F$ and \Cref{asm: stability}, we have that
	\begin{equation} 
	\label{eq: ana_1}
	\norm{e^n_i} \le  L \sum_{j \in \cV} \sum \limits_{m = n - \tau_{ij}(n)}^{n-1} \norm{x_j ^{m+1} - x_j ^m} \le C \sum_{j \in \cV} \sum \limits_{m = n - \tau_{ij}(n)}^{n-1} a(m),
	\end{equation}
	for a sample path dependent constant $C > 0$.
	We will now show that
	\begin{equation}
	\label{eq: ana_claim}
	\lim\limits_{n \to 0 }  \left( \sum \limits_{m = n - \tau_{ij}(n)}^{n-1} a(m) \right) = 0,
	\end{equation}
	which will imply that $\lim\limits_{n \to \infty} \norm{e^n_i} = 0$.
	
	By \Cref{asm: step_size}, $a(n) \in \cO(n^{-\frac{1}{p}})$. 
	Hence, there are constants $c>0$ and $N \in \N$, such that 
	\begin{equation}
	\label{eq: ana2}
	a(n) \le c n^{-\frac{1}{p}} \text{ for all } n\ge N.
	\end{equation}
	Also by \Cref{asm: step_size}, there is some $\overline{\tau}$  that stochastically dominates all $\tau_{ij}(n)$, with $\Ew{\overline{\tau}^p} < \infty$. Now fix $\varepsilon \in (0,1)$. By \Cref{lem:finite_CDFsum} we have that
	\begin{equation}
	\sum_{n=0}^{\infty} \Pr{\tau_{ij}(n) > \varepsilon n^{\frac{1}{p}}} < \infty. 
	\end{equation}
	It now follows from the Borel-Cantelli Lemma that $\Pr{\tau_{ij}(n) > \varepsilon n^{\frac{1}{p}} \text{ i.o. }} = 0.$
	Hence, there is sample path dependent $N(\varepsilon) \in \N$, such that
	\begin{equation}
	\label{eq: ana3}
	\tau_{ij}(n) \le \varepsilon n^{\frac{1}{p}} \qquad \forall \, n\ge N(\varepsilon).
	\end{equation} 
	\Cref{eq: ana2,eq: ana3} therefore yield that
	\begin{equation}
	\sum \limits_{m = n - \tau_{ij}(n)}^{n-1} a(m) \le c \sum \limits_{m = n - \varepsilon n^{\frac{1}{p}} }^{n-1}  m^{-\frac{1}{p}}
	\end{equation}
	for all $n$ with $n \ge N(\varepsilon)$ and $n - \varepsilon n^{\frac{1}{p}} \ge N$.
	Finally, using the monotonicity of $n^{-\frac{1}{p}}$, we have 
	\begin{equation}
	\sum \limits_{m = n - \varepsilon n^{\frac{1}{p}} }^{n-1}  m^{-\frac{1}{p}} \le \varepsilon n^{\frac{1}{p}} (n - \varepsilon n^{\frac{1}{p}} )^{-\frac{1}{p}} = \varepsilon (1- \varepsilon n^{\frac{1}{p} - 1} )^{-\frac{1}{p}} \xrightarrow[n \to \infty]{}
	\begin{cases}
	\varepsilon& \quad p\in (1,2),\\
	\frac{\varepsilon}{1-\varepsilon}& \quad p=1.
	\end{cases} 
	\end{equation} 
	Hence,
	\begin{equation}
	\limsup\limits_{n\rightarrow \infty} \left( \sum \limits_{m = n - \tau_{ij}(n)}^{n-1} a(m) \right) \le
	\frac{c\varepsilon}{1-\varepsilon}
	\end{equation}
	and \eqref{eq: ana_claim} follows, since the choice of $\varepsilon$ is arbitrary. 
\end{proof}

In \cite[Theorem 1]{ram2022_tac} we proved the convergence of \Cref{alg: 1} for $\tau_{ij}(n) = 0$ for all $n \in \N_0$. The following theorem is now an immediate consequence of this result and \Cref{lem:vanishing_errors}.
\begin{theorem}
	\label{thm:convergence}
	Under \Cref{asm: objective,asm: stability,asm: step_size,asm: additive_error}, we have that \Cref{alg: 1} converges almost surely to a $\lambda$-neighbourhood of the set of stationary points of F, where $\lambda$ is the almost sure bound of the additive errors according to $\Cref{asm: additive_error}$.
\end{theorem}

\section{A new set of network conditions for distributed optimization} 
\label{sec:network_conditions}

In the previous section, we presented a convergence proof for \Cref{alg: 1} under the network assumption \Cref{asm: step_size}.1. This assumption directly requires that some $p$-th moment of all AoI variables is bounded. However, the distribution of all AoI variables will typically be the consequence of direct agent to agent communication.
We are therefore interest in more concrete conditions on the network and the agent communication that imply the required AoI moment conditions. To achieve this, this section introduces a network model and associated assumptions to verify \Cref{asm: step_size}.1.

\subsection{Network model}
\label{sec:net_model}

Recall that \Cref{alg: 1} requires that the agents exchange their local variables $x_i^n$ over a network.  The network and an associated communication protocol should allow that local variables $x_i^n$ can frequently spread across the network and reach every agent.
We will now introduce a network model were the agents try to exchange their local belief vectors $\hat{X}_i^n$. The agents therefore try to share their latest available version of all other agents local variable with other agents. This might potentially flood the network with data, however, there well known protocols to reduce the number of possibly redundant transmissions \cite{lim2001flooding}.

We assume a time-varying network (\Cref{def:tv_network})
\begin{equation}
\{(\cV,\cE^n)\}_{n \in \N_0},
\end{equation}
which is a sequence of directed graphs. Each agent is in one-to-one correspondence with one node in the graph. For every time step $n\in \N_0$, an edge $(i,j) \in \cE^n$ represents the event that agent $i$ successfully exchanges its local believe vector $\hat{X}_i^n$ during time slot $n$ with agent $j$.
We denote this event by $A^n_{ij}$. Therefore, the sequence of directed graphs and the sequences of events $\{A^n_{ij}\}_{n \in \N_0}$ are in one to one correspondence: An edge $(i,j) \in \cE^n$ if and only if the event $A^n_{ij}$ occurs. An edge therefore does not represent the possibility for communication, but the actual event of communication. 

One may add additional complexity to the model, e.g. using a graph that represents the possibility for communication. Additionally, the model may be extended to scenarios where multiple successive events $A_{ij}^n$ need to occur to guarantee the exchange of a single realization of a believe vector $\hat{X}^n_i$. This might be necessary if the dimension of $\hat{X}^n_i$ is very large and/or the network bandwidth is small. 

Note that although we defined the events $A^n_{ij}$ for all $(i,j) \in \cV \times \cV$, some of those events might never occur over the whole time horizon. We will especially do not require that all agents communicate directly! However, at least some of the events $A^n_{ij}$ should occur ``frequently'' enough such that the time-varying network satisfies certain connectivity properties. This will be formulated in \Cref{subsec:SSC} with \Cref{asm: network-A}.

The formulation of the time-varying communication network using the edge events  $A^n_{ij}$ has several advantages. The model allows for an underlying time-varying graph that may be the consequence of an network scheduling algorithm or the physical dynamics of the agents themselves.
Each event $A_{ij}^n$ can be represented as a multistage process. For example, (i) the availability of a channel, (ii) the use of an access protocol given the availability of a channel, (iii) the success of the transmission given the successful channel access. In general, the event-based formulation appears to be very convenient for analysis.

In the next two subsection, we will formulate our assumptions for the time-varying network $\{(\cV,\cE^n)\}_{n \in \N_0}$ using the events $A_{ij}^n$.

\subsection{Stochastic strong connectedness}
\label{subsec:SSC}
The following assumption formalizes our required network connectivity property.

\begin{assumption}[Network connectivity assumption]
	\label{asm: network-A}
	We assume that the time-varying network is $(\varepsilon,\kappa)$-stochastically strongly connected (SSC) (\Cref{def:SSC}) for some $\varepsilon \in (0,1)$ and some $\kappa \in \N_0$.
\end{assumption}

Using the events $A^n_{ij}$, a  $(\varepsilon,\kappa)$-SSC network requires there exists a strongly connected graph $(\cV, \cE)$, such that for all  $n \in \N_0$ and for all $(i,j) \in \cE$, we have
\begin{equation}
\Pr{\bigcup_{k=n}^{n+\kappa} A^n_{ij}} \ge \varepsilon.
\end{equation}
A $(\varepsilon, \kappa)$-SSC network therefore requires that there are some pairs of agents $(i,j) \in \cE$ that can communicate directly at least ones in every time-interval of the form $[n, n+\kappa]$ with positive probability $\varepsilon$. Notice that SSC does not require direct communication between every pair of agents. The only agents that do communicate are those given in the set $\cE$. A SSC network reflects our intuition of a non-degenerate communication network. Some agents can ``frequently'' exchange information with positive probability and information can spread across the network since $\cE$ is strongly connected. 

Note that a network that is SSC does not imply guaranteed transmissions periodically. 
We will see shortly that SSC is significantly weaker that plain guaranteed periodic communication. With stochastic strong connectivity we can not draw any conclusions about the dependency of events in the network. On the other hand, assuming guaranteed periodic communication does imply a strong form of dependency decay as shown in \Cref{sec: comparison}. Recall that our objective is to verify \Cref{asm: step_size}.1. However, using SSC alone is not sufficient to even guarantee the existence of a dominating random variable as required in \Cref{asm: step_size}.1. The next subsection therefore formulates dependency decay conditions using strong mixing (\Cref{def:mixing}).

\subsection{Network dependency decay}
\label{subsec:dependency_decay}

Recall that our time-varying network is given by a sequence of directed graphs $\{(\cV,\cE^n)\}_{n \in \N_0}$. The sequence is in one-to-one correspondences with events $A^n_{ij}$ that represent the presence of an edge at time $n$. We will now formulate a dependency decay assumption based on the notion of strongly mixing processes. We can then show that the AoI variables $\tau_{ij}(n)$ associated with a $(\varepsilon, \kappa)$-SSC network satisfies specific moment conditions depending on the assumed rate at which dependency decays in the network.

\begin{assumption}[Dependency decay assumption]
	\label{asm: network-B}
	We assume that the time-varying network is such that there is some $\eta \ge0$ such that each process $\Ind{\bigcup_{k=n}^{n+\eta} A^n_{ij}}$ is $p$-strongly mixing (\Cref{def: mixing_process}) for some $p \in [1,2)$.
\end{assumption}



With this assumption we do not require that the dependency of subsequent events $A_{ij}^n$ does decay at any specific rate. However, there should be an interval size $\eta>0$, such that the dependency of subsequent union events $\bigcup_{k=n}^{n+\eta} A^n_{ij}$ decays sufficiently fast. Notice that \Cref{asm: network-B} is a dependency decay assumption for the network processes $\Ind{\bigcup_{k=n}^{n+N}A^n_{ij}}$ associated with all network edges $(i,j) \in \cV$. However, we actually only require the assumption for those edges $(i,j) \in \cE$ in an edge set $\cE$ according to \Cref{asm: network-A}. \emph{Additionally, notice that we do not require any form of independence or dependency decay between transmissions over different edges.} The reason for this is \Cref{lem: key_lemma}. The lemma will show that the existence of a dominating random variable for the AoI variables is in a natural way a transitive property of the network. 

In this work we don't give a recipe to verify \Cref{asm: network-B}. However, we will see in the next subsection that the standard assumptions in the distributed optimization literature all imply \Cref{asm: network-B}. Another set of examples where \Cref{asm: network-B} is also directly satisfied are scenarios where the network events $A^n_{ij}$ are driven by a geometrically ergodic Markov process \cite{davydov1974mixing,bradley2005basic}. Of course, it can be comparatively difficult to verify this in practice. However, traditionally and also more recently it has been quite common to model network fading channels by finite Markov chains \cite{wang1995finite,bianchi2000performance,pimentel2004finite,lin2015finite,boban2016modeling}. We further discuss the verifiability of \Cref{asm: network-B} in \Cref{sec:conclusion}.

\subsection{Comparison of \Cref{asm: network-A,asm: network-B} to assumptions in the literature}
\label{sec: comparison}

In this subsection we show that the typical network assumptions in the literature imply \Cref{asm: network-A,asm: network-B}.

First, consider the network models in \cite{ram2022_tac, bastianello2020asynchronous, Lei2018-ag, koloskova2020unified}. It is easy to check that network models imply the following properties:
\begin{enumerate}
	\item There is a strongly connected graph $(\cV, \cE)$ and some $\varepsilon >0$, such that $ \Pr{A_{ij}^n} > \varepsilon $
	for all $n \in \N_0$ and for all $(i,j) \in \cE$.
	\item The events $A_{ij}^n$ are independent for different time-steps or different edges.
\end{enumerate}

Independence is particularly unrealistic for wireless communication systems, since transmission that occur close in time, space, frequency or code can be highly correlated. 
Notably, this assumptions do not show any trade off between the choice of the step size sequence $a(n)$ and some network related property. Hence, there is no trade of between the growth of the AoI variables and the choice of the step size sequence. In fact, it is easy to show that under this assumptions \emph{all} moments of \emph{all} $\tau_{ij}(n)$ are bounded, see \Cref{sec:construction} \Cref{ex:iid}.

We can now show that the above properties imply \Cref{asm: network-A,asm: network-B}. \Cref{asm: network-A} is directly satisfied for $\kappa = 1$. Define the $\sigma$-algebras
\begin{equation}
	\cF_l^m \coloneqq \sigma\left(A^n_{ij} \mid l \le n \le m;  i,j\in \cV \right).
\end{equation}
Then \Cref{asm: network-B} holds trivially, since the independence of the events $A_{ij}^n$ implies that
\begin{equation}
	\lvert \Pr{A \cap B} - \Pr{A}\Pr{B} \rvert = 0.
\end{equation}
for $A \in  \cF_0^l$ and $B \in  \cF_{l+n}^\infty$ for all $l,n \in \N$. Hence, the mixing coefficients $\alpha_{ij}(n)$ for each process $A_{ij}^n$ satisfies $\alpha_{ij}(n) = 0$ for every $n\ge 0$.


Second, consider the time-varying network in \cite{xu2017distributed, nedic2017achieving, wang2019distributed, Scutari2019-sa, aybat2019distributed, Yu2020-gq,kovalev2021adom}. The authors assume that their network is $M$-strongly connected. Hence, they assume guaranteed periodic communication.
\Cref{asm: network-A} is therefore directly satisfied by choosing $\kappa = M$. Then $\Pr{\bigcup_{k=n}^{n+\kappa} A^k_{ij}} = 1 $ for all $n \in \N_0$. \Cref{asm: network-B} is also directly satisfied by choosing $\eta = M$. To see this, fix any $n,m \ge0$ with $m \not=n$. Then 
\begin{equation}
	\Pr{\left(\bigcup_{k=n}^{n+\eta} A^k_{ij}\right) \cap  \left(\bigcup_{k=m}^{m+\eta} A^k_{ij}\right)} = 1,
\end{equation}
since the intersection of almost sure events is an almost sure event. Therefore,
\begin{equation}
	\Pr{\left(\bigcup_{k=n}^{n+\eta} A^k_{ij}\right) \cap  \left(\bigcup_{k=m}^{m+\eta} A^k_{ij}\right)} - \Pr{\bigcup_{k=n}^{n+\eta} A^k_{ij}} \Pr{\bigcup_{k=m}^{m+\eta} A^k_{ij}} = 0
\end{equation}
and \Cref{asm: network-B} follows.

We have therefore shown that the network models in the literature satisfy \Cref{asm: network-A,asm: network-B}. Moreover, \Cref{asm: network-A,asm: network-B} are significantly weaker, since they do not require independent communication or guaranteed periodic communication, but merely asymptotic independence. 

\section{Stochastic dominance properties of AoI for Time-Varying Networks}
\label{sec: analysis}

In this section, we show that \Cref{asm: network-A,asm: network-B} imply \Cref{asm: step_size}.1. 
Recall that the AoI variables $\tau_{ij}(n)$, as defined in \Cref{sec: problem}, are now a consequence of the network model formulated in \Cref{sec:net_model}. Each agent tries to send its local believe vector $\hat{X}_i^n$ (\Cref{eq:believe_vector}) to some other agents. A successful transmission to some other agent $j$ is represented by an edge $(i,j) \in \cE^n$ of the time-varying network $\{(\cV,\cE^n)\}_{n \in \N_0}$ or equivalently by the event $A_{ij}^n$.

Recall that \Cref{asm: step_size}.1 requires finite moment properties of a random variable that stochastically dominates (\Cref{def: stochastic_dominance}) all $\tau_{ij}(n)$. The following definition will be useful to formulate our main result and the subsequent proof.
\begin{definition}
	\label{def: stoch_dom_p} 
	We say an AoI variable $\tau_{ij}(n)$ is \textbf{stochastically dominated with finite $\boldsymbol{p}$-th moment} for some $p\ge 0$, if there exists a non-negative integer-valued random variable $\overline{\tau}$ that stochastically dominates all $\tau_{ij}(n)$ for and all $n \in \N_0$ with $\Ew{\overline{\tau}^p} < \infty$.
\end{definition}


The following theorem formulates the main result of this section.
\begin{theorem}
	\label{thm: 0}
	Let  $\{(\cV,\cE^n)\}_{n \in \N_0}$ be a time-varying network that is $(\varepsilon, \kappa)$-SSC (\Cref{def:SSC}) with associated strongly connected graph $(\cV, \cE)$. If for each $(i,j) \in \cE$, there is some $\eta\in \N_0$, such that the process $\Ind{\bigcup_{k=n}^{n+\eta} A^n_{ij}}$ is $p$-strongly mixing (\Cref{def: mixing_process}) for some $p\ge 0$, then all AoI variables $\tau_{ij}(n)$ are stochastically dominated by a single random variable with finite $p$-th moment.
\end{theorem}

Stochastic dominance with finite $0$-th moment corresponds to the mere existence of a dominating random variable without any necessary moment condition. \Cref{thm: 0} shows a more general result as it would be required for the convergence of \Cref{alg: 1}. It is shown for all $p \in [0,\infty)$. The following corollary is now immediate and requires \Cref{thm: 0} for $p \in [1,2)$.

\begin{corollary}
	Under \Cref{asm: objective,asm: stability,asm: additive_error,asm: network-A,asm: network-B}, we have that \Cref{alg: 1} converges almost surely to a $\lambda$-neighbourhood of the set of stationary points of F, where $\lambda$ is the almost sure bound of the additive errors according to $\Cref{asm: additive_error}$.
\end{corollary}
\begin{proof}
	Under \Cref{asm: network-A} and \ref{asm: network-B}, it follows from \Cref{thm: 0} that \Cref{asm: step_size}.1 holds for some $p \in [1,2)$. We can then choose a step size sequence $a(n)$ that is not summable, but square summable with $a(n) \in \cO(n^{-\frac{1}{p}})$ and therefore also satisfy \Cref{asm: step_size}.2. The requirements of \Cref{thm:convergence} are therefore satisfied and the statement of the corollary follows.
\end{proof}

The rest of this section is devoted to the proof of \Cref{thm: 0}. We begin by describing a general construction/recipe to establish the stochastic dominance properties for AoI variables of time-varying networks. In addition, we illustrate the recipe for the scenario where the edge events $A^n_{ij}$ are independent. Afterwards, we give the proof of \Cref{thm: 0}. Before proceeding, we show a preliminary property of the AoI variables for an $(\varepsilon, \kappa)$-SSC network.
\begin{lemma}
	\label{lem: prelim}
	Let  $\{(\cV,\cE^n)\}_{n \in \N_0}$ be a time-varying network that is $(\varepsilon, \kappa)$-SSC with associated strongly connected graph $(\cV, \cE)$, then for all $(i,j)\in \cE$ we have
	\begin{equation}
	\Pr{\tau_{ij}(n) > m} < \varepsilon, \qquad \forall \, m,n \ge \kappa,
	\end{equation} 	
\end{lemma}
\begin{proof}
	First, we have have $\Pr{\tau_{ij}(n) > m} = 0 $ for $m\ge n$, since $\tau_{ij}(n)\le n$. We therefore concentrate on $m < n$.
	Fix  $(i,j) \in \cE$, i.e. $i$ and $j$ are agents that can communicate directly. Observe that successful direct communication from $i$ to $j$ during any time interval of the form $[n-m+1, n]$ implies that the AoI at time $n$ is less than $m$. In other words, we have the following inclusion
	\begin{equation}
	\label{eq:prelimlemma_inc}
	\{\bigcup_{l=n-m+1}^{n} A^l_{ij}\} \subset  \{\tau_{ij}(n) \le m\}.
	\end{equation}
	Since the network is $(\varepsilon, \kappa)$-SSC, we have that 
	\begin{equation}
	\Pr{\tau_{ij}(n) \le m} \ge \Pr{\bigcup_{l=n-m+1}^{n} A^l_{ij}} \ge \varepsilon, \qquad \forall \, n > m \ge  M
	\end{equation}
	The complementary event of the previous expression therefore concludes the proof of the lemma.
\end{proof}

\subsection{A construction to establish stochastic dominance properties}
\label{sec:construction}
We now describe a general construction to establish the stochastic dominance properties with some finite $p$-th moment (\Cref{def: stoch_dom_p}) for an AoI variable $\tau_{ij}(n)$. The idea is to find a uniform upper bound $u:\N_0 \to \R_{\ge0}$, such that $$\Pr{\tau_{ij}(n) >m} \le u(m)$$ for all $m \ge N$ independent of $n\in \N_0$ for some $N\in \N_0$ and $\lim\limits_{m\to \infty} u(m) = 0$. We can now use this bound to define the CDF of a new random variable. Since $\lim\limits_{m\to \infty} u(m) = 0$ there is some $M \in \N_0$, such that $u(m) \le 1$ for all $m\ge M \ge N$. Now define a non-negative integer-valued random variable $\overline{\tau}_{ij}$ by describing its CDF (more precisely its complementary CDF) as follows:
\begin{align}
\Pr{\overline{\tau}_{ij} > m} &= 1, \quad  & 0\le  m < M, \\
\Pr{\overline{\tau}_{ij} > m} &= u(m) , \quad & m\ge M. 
\end{align}
By definition $\overline{\tau}_{ij}$ stochastically dominates all $\tau_{ij}(n)$ for all $n\in \N_0$. Moreover, if $$\sum_{m=0}^\infty ((m+1)^{p}-m^p) u(m) < \infty$$ for some $p>0$, then it will follow from \Cref{eq:moment_eq} that $\tau_{ij}(n)$ is stochastically dominated with finite $p$-th moment.

As the next step, we describe how we can find a function $u(m)$ for the above construction. Consider a $(\varepsilon, \kappa)$-SSC network. Let $(\cV,\cE)$ be the strongly connected graph associated with the $(\varepsilon, \kappa)$-SSC network and fix an edge $(i,j) \in \cE$.  Let $\Delta(m)$ be an increasing sequence in $\N$, with $\lim\limits_{m \to \infty} \Delta(m) = \infty$. Now for each $n,m \in \N_0$ use this sequence to define time indices
\begin{align}
\label{eq:time_ind}
n_1 \coloneqq n-m + \Delta(m),  \qquad n_k &\coloneqq n_{k-1} + 2\Delta(m)
\end{align}
as long as $n_k \le n$. Let $L(m)$ be the number of constructed time indices and observe that
\begin{equation}
\label{eq:bound_u}
\Pr{\tau_{ij}(n) >m} \le \Pr{ \bigcap_{k=1}^{L(m)}  \{\tau_{ij}(n_k) > \Delta(m) \}}.
\end{equation}
This follows since $\tau_{ij}(n) >m$ implies $\tau_{ij}(n_k) > \Delta(m)$ for all $k \in \{1,\ldots, L(m)\}$ by the very construction of the time indices $n_k$. In general, we can now derive $u(m)$ as an upper bound to the right-hand side in \eqref{eq:bound_u}, which we illustrate immediately for case of independent network communication, i.e. were the events $A_{ij}^n$ are independent. For the case of dependent network communication, this will be formulated in \Cref{lem: aoi} in the next section.

\begin{example}[Independent network communication]
	\label{ex:iid}
Let $(\cV, \cE)$ be the strongly connected graph associated with a $(\varepsilon, \kappa)$-SSC network and consider an edge $(i,j) \in \cE$. Using the exemplary network independence and \Cref{lem: prelim}, we have from  \eqref{eq:bound_u} that
\begin{equation}
\label{eq:illustration_iid}
\Pr{\tau_{ij}(n) >m} \le \prod_{k=1}^{L(m)} \Pr{ \tau_{ij}(n_k) > \Delta(m)} < \varepsilon^{L(m)},
\end{equation}
for all $m$ large enough such that $\Delta(m) \ge \kappa$. Now define 
$$u(m) \coloneqq \varepsilon^{L(m)}, \quad \Delta(m) \approx \sqrt{m} $$
and hence $L(m) \approx \nicefrac{\sqrt{m}}{2}$. 
The construction described above then yields a dominating random variable $\overline{\tau}$ for all $\tau_{ij}(n)$ for all $n\in \N_0$. It is now easy to verify that
$$ \Ew{\overline{\tau}^p} \le \sum_{m=0}^\infty ((m+1)^{p}-m^p) u(m) \approx \sum_{m=0}^\infty((m+1)^{p}-m^p) \varepsilon^{\nicefrac{\sqrt{m}}{2}} < \infty$$
\emph{for all $p\ge0$}, since the series is a version of a weighted geometric series. We have therefore established that with independent communication, each AoI variable $\tau_{ij}(n)$ with $(i,j) \in \cE$ is stochastically dominated with finite $p$-th moment \emph{for every $p\ge0$.} This underlines how strong the assumption of independent communication is. 
\end{example}

\subsection{Proof of \Cref{thm: 0}}

In the previous example, we used the independence of the edge events $A_{ij}^n$ to establish a uniform upper bound for $\Pr{\tau_{ij}(n) >m}$ with geometric decay. Recall that $\Delta(m)$ was used in \eqref{eq:time_ind} to define the time indices $n_k$, such that $n_{k}-\Delta(m) - n_{k-1} = \Delta(m)$. 
Now consider the case where the edge events are not independent but merely mixing. We will see that we can then find a new upper bound to  \eqref{eq:bound_u}, such that
\begin{equation}
\Pr{\tau_{ij}(n) >m} \le \Pr{ \bigcap_{k=1}^{L(m)}  \{\tau_{ij}(n_k) > \Delta(m) \}} \le \varepsilon^{L(m)} + error(\Delta(m)).
\end{equation}
with an error term $error(\Delta(m))$ due to the non independence. 

Now, if the mixing coefficients associated with processes $\Ind{\bigcup_{k=n}^{n+\eta} A^n_{ij}}$ decay rapidly enough, we expect that $error(\Delta(m))$ decays sufficiently, such that the new upper bound still satisfies some summability properties and hence allows that we establish stochastic dominance properties. The following lemma makes this intuition precise. We establishes the stochastic dominance property of order $p\ge0$ for those network edges $(i,j)$ that ensure that the network is $(\varepsilon, \kappa)$-SSC.

\begin{lemma}
	\label{lem: aoi}
	Let  $\{(\cV,\cE^n)\}_{n \in \N_0}$ be a time-varying network that is $(\varepsilon, \kappa)$-SSC (\Cref{def:SSC}) with associated strongly connected graph $(\cV, \cE)$. If for any $(i,j) \in \cE$ the process $\Ind{\bigcup_{k=n}^{n+\eta} A^n_{ij}}$ is $p$-strongly mixing (\Cref{def: mixing_process}) for some $p\ge 0$ and some $\eta \in \N_0$, then $\tau_{ij}(n)$ is stochastically dominated with finite $p$-th moment (\Cref{def: stoch_dom_p}).
\end{lemma}
\begin{proof}
	Fix an edge $(i,j) \in \cE$. The theme of the proof is to establish a uniform upper bound to the complementary CDF of $\tau_{ij}(n)$ independent of $n$, such that the construction from \Cref{sec:construction} yields the required dominating random variable.
	
	\emph{Step 1 (Reduction to $\eta =0$)}: The $p$-strongly mixing property of the network guarantees mixing of the process $\Ind{\bigcup_{k=n}^{n+\eta} A^n_{ij}}$ for some $\eta \in \N_0$.  W.l.o.g. we can assume that $\eta=0$. This is justified as follows. Lets denote by $\tau^\eta_{ij}(k)$ a new random variable that captures the time, since the last interval of the form $[m\eta,(m+1)\eta]$ with at least one successful transmission from $i$ to $j$. The case $\eta= 0$ then yields the conclusion of the Lemma for $\tau^\eta_{ij}(k)$, i.e. there will be random variable $\overline{\tau}^\eta_{ij}$ that stochastically dominates all $\tau^\eta_{ij}(k)$ with $\Ew{(\overline{\tau}^\eta_{ij})^p} < \infty$.
	For any $k \ge 0 $ and $n \in  \{k\eta, (k+1)\eta\}$, we have 
	$\tau_{ij}(n) \le \eta(\tau_N(k) + 1) $.
	Therefore,
	\begin{equation}
	\Pr{\tau_{ij}(n) > m} \le \Pr{N (\tau^\eta_{ij} (k) + 1) > m} \le \Pr{\eta (\overline{\tau}^\eta_{ij} + 1) > m}
	\end{equation}
	and $\Ew{\eta^p (\overline{\tau}^\eta_{ij} + 1)^p} < \infty$ by Minkowski's inequality. Hence, $ \eta(\tau_N(k) + 1)$ would be the required dominating random variable for $\tau_{ij}(n)$ and we may therefore assume $\eta=0$.
	
	\emph{Step 2 (Initial CDF bound)}: Fix $m \in \N_0$ and recall the definition of $\Delta(m)$ and the associated sequence $n_k$ for each $n \in \N_0$ from \Cref{sec:construction}. We have
	\begin{equation}
	\label{eq:bound_u2}
	\Pr{\tau_{ij}(n) >m} \le \Pr{ \bigcap_{k=1}^{L(m)}  \{\tau_{ij}(n_k) > \Delta(m) \}}. \tag{\ref{eq:bound_u} recalled}
	\end{equation}
	\emph{With a slide abuse of notation we will now refer with $\tau_{ij}(n)$ to the age of information associated with the direct information exchange from $i$ to $j$. The age of information associated with direct information exchange by definition stochastically dominates the actual AoI. Without this step we would technically require a stronger mixing requirement, specifically, one for the events generated by all $A_{ij}^n$ and not only for the events generated by $A_{ij}^n$ for the pair $(i,j)$. Note that \Cref{lem: prelim} also directly holds for this case, since we anyway used the direct information exchange to prove it.}

	We will now establish an upper bound to \eqref{eq:bound_u2} using that $\Ind{A^n_{ij}}$ is $p$-strongly mixing. For this, define the following sub-$\sigma$-algebras generated by the events $A^n_{ij}$:
	\begin{equation}
	\cF_l^s \coloneqq \sigma\left( A^n_{ij} \mid l \le n \le s\right), \quad l \in \N_0, s\in \N_0 \cup \{\infty\}.
	\end{equation}
	The important generated events are, whether the AoI variables at some time step $s \in \N_0$ exceed a threshold $l \in \N_0$, i.e. whether $\{\tau_{ij}(s) > l\}$.
	Since the event $\{\tau_{ij}(s) > l\}$ is generated by the events $A^k_{ij}$ with $k \in \{s-l+1,\ldots,s-1,s\}$,we have that
	\begin{equation}
	\{\tau_{ij}(s) > l\} \in \cF^s_{s-l+1}.
	\end{equation}
	For this, we required the reduction to age of information associated with direct information exchange.
	It then follows by definition of the time indices $n_k$ that 
	\begin{equation}
	\{\tau_{ij}(n_{L(m)}) > \Delta(m)  \} \in \cF^{n_{L(m)}}_{n_{L(m)}-\Delta(m) + 1} \subset \cF^{\infty}_{n_{L(m)}-\Delta(m)}
	\end{equation}
	and
	\begin{equation}
	\{\tau_{ij}(n_{k}) > \Delta(m)  \} \in \cF^{n_{k}}_{n_{k}-\Delta(m) + 1} \subset \cF^{n_{L(m)-1}}_0
	\end{equation}
	for every $k \in \{1, \ldots, L(m)-1\}$. Hence,
	\begin{equation}
	\bigcap_{k=1}^{L(m)-1} \{ \tau_{ij}(n_k) > \Delta(m) \} \in \cF^{n_{L(m)-1}}_0.
	\end{equation}
	
	By construction of the indices $n_k$, we have $n_{L(m)}-\Delta(m) - n_{L(m)-1} = \Delta(m)$. The strong mixing property of the process $\Ind{A^n_{ij}}$ therefore implys that
	\begin{equation}
	\begin{split}
	\label{eq: recursion}
	\Pr{ \bigcap_{k=1}^{L(m)}  \{\tau_{ij}(n_k) > \Delta(m) \}} &\le \Pr{\{\tau_{ij}(n_{L(m)}) > \Delta(m) \}}\Pr{\bigcap_{k=1}^{L(m)-1} \{ \tau_{ij}(n_k) > \Delta(m) \}} \\ & \qquad + \alpha(\Delta(m)),
	\end{split}
	\end{equation}
	where $\alpha(n)$ are the mixing coefficients associated with the process $\Ind{A^n_{ij}}$.
	It now follows from \Cref{lem: prelim} that $\Pr{\tau_{ij}(n_k) > \Delta(m) } < \varepsilon$ for $\Delta(m) \ge \kappa$, since the network is $(\varepsilon, \kappa)$-SSC. Hence, 
	\begin{equation}
	\label{eq: recursion2}
	\Pr{ \bigcap_{k=1}^{L(m)}  \{\tau_{ij}(n_k) > \Delta(m) \}} \le \varepsilon\Pr{\bigcap_{k=1}^{L(m)-1} \{ \tau_{ij}(n_k) > \Delta(m) \}}  + \alpha(\Delta(m)).
	\end{equation}
	Applying \eqref{eq: recursion} and \eqref{eq: recursion2} successively yields:
	\begin{align}
	\Pr{\tau_{ij}(n) >m} &\le \prod_{k=1}^{L(m)} \Pr{\{\tau_{ij}(n_{k}) > \Delta(m) \}}  + \sum_{k=1}^{L(m)-1}
	\varepsilon^{k-1} \alpha(\Delta(m)) \label{eq: uniform_bound1}\\ &\le \varepsilon^{L(m)} + \frac{1}{1-\varepsilon} \alpha(\Delta(m)) \label{eq: uniform_bound2}.
	\end{align}
	for $\Delta(m) \ge \kappa.$ 
	
	For $p=0$, we can now apply the construction presented in \Cref{sec:construction} with the bound \eqref{eq: uniform_bound2} to obtain a dominating random variable. Here we may choose $\Delta(m)$ as in \Cref{ex:iid}. For $p>0$ it is now crucial to choose $\Delta(m)$, such that both terms in \eqref{eq: uniform_bound2} decay rapidly enough to obtain the required stochastic dominance property with finite $p$-th moment. However, it turns out that the bound \eqref{eq: uniform_bound2} is only sufficient to achieve this for all $q<p$, due to the merely geometric decay of the first term. The next step therefore uses \eqref{eq: uniform_bound2} to obtain a better upper bound for \eqref{eq: uniform_bound1}.
	
	\emph{Step 3}: 
	To improve the CDF bound for $p>0$, we use that $\sum_{m=0}^{\infty} m^{p-1} \alpha(m)$ is summable. It then follows that for $p>1$ we have
	\begin{equation}
		\alpha(m) \in \cO(m^{-(p-1)})
	\end{equation}
	and for $0< p\le 1$ we have
	\begin{equation}
	\alpha(m) \in \cO(m^{-p}),
	\end{equation}
	since for this case $m^{p-1}\alpha(m)$ is guaranteed to be decreasing as $p-1<0$.
	Both cases show that there is a constant $c$ and some $\tilde{\mu}>0$, such that
	\begin{equation}
	\alpha(\Delta(m)) \le C(\Delta(m))^{-\tilde{\mu}}
	\end{equation}
	for sufficiently large $m$.
	With \eqref{eq: uniform_bound2} it then follows that
	\begin{equation}
	\Pr{\tau_{ij}(n) >m} \le  \varepsilon^{L(m)} + c(\Delta(m))^{-\tilde{\mu}}.
	\end{equation}
	for sufficiently large $m$.
	Since the first term above is exponential and the second is rational, we can find a new $\mu >0$, such that
	\begin{equation}
	\Pr{\tau_{ij}(n) > m} <  m^{-\mu}
	\end{equation}
	for $m$ sufficiently large. For this, one may again choose $\Delta(m) \approx \sqrt{m}$.

	\emph{Step 4 - (Verifying the stochastic dominance property with finite $p$-th moment)}: 
	
	We now insert the CDF bound from step 4 in \eqref{eq: uniform_bound1} and obtain
	\begin{equation}
	\begin{split}
	\label{eq: uniform_bound3}
	\Pr{\tau_{ij}(n) >m} \le  \Delta(m)^{-\mu L(m)} + \frac{1}{1-\varepsilon} \alpha(\Delta(m))
	\end{split}
	\end{equation}
	for $m$ sufficiently large. 
	Now choose $\delta \in (0,1)$, such that
	\begin{equation}
		\mu(\frac{1}{4\delta} -1) \ge p+1
	\end{equation}
	and then choose $\Delta(m) = \lceil \delta m \rceil$. We choose this to guarantee the required summability property of the first term in \eqref{eq: uniform_bound3}, since
	\begin{equation}
		L(m) = \lfloor \frac{m}{2\Delta(m)} \rfloor \ge  \frac{1}{4\delta} - 1
	\end{equation}
	for $m\ge \frac{1}{2\delta}$. Hence, we have 
	\begin{align}
	\Pr{\tau_{ij}(n) >m} &\le  (\delta m)^{-(p+1)} + \frac{1}{1-\varepsilon} \alpha(\lceil \delta m \rceil)
	\end{align}
	for $m\ge \frac{1}{2\delta}$. 
	
	Now define $$u(m) \coloneqq (\delta m)^{-(p+1)} + \frac{1}{1-\varepsilon} \alpha(\lceil \delta m \rceil)$$ and apply the construction presented in \Cref{sec:construction}. This yields a non-negative integer-valued random variable $\overline{\tau}_{ij}$ that stochastically dominates $\tau_{ij}(n)$ for all $n\in \N$. Moreover, we have $\Ew{\overline{\tau}_{ij}^p} < \infty$, if 
	\begin{equation}
	\sum_{m=0}^{\infty} ((m+1)^p-m^p) u(m) < \infty.
	\end{equation}
	The first part of the series is finite, since
	\begin{equation}
		\sum_{m=1}^{\infty} ((m+1)^p-m^p) (\delta m)^{-(p+1)} \le \frac{2^p}{\delta^{p+1}} \sum_{m=1}^{\infty} m^{-2} < \infty,
	\end{equation}
	where we used that $((m+1)^p-m^p) \le 2^p m^{p-1}$ for $m \in \N$.
	
	For the second part of the series, note that $\alpha(n)$ is by construction a monotonically decreasing function from $\N_0$ to $[0, \frac{1}{4}]$ \cite{bradley2005basic}. Now extend 	$\alpha(n)$ by linear interpolation to a monotonically decreasing function from $\R_{\ge0}$ to $[0, \frac{1}{4}]$. Then for all $m \in \N_0$, we have $\alpha(\lceil \delta m \rceil) \le \alpha(\delta m)$ by monotonicity. Hence the second part is finite, since
	\begin{equation}
	\begin{split}
	\sum_{m=1}^{\infty} ((m+1)^p-m^p)\alpha(\delta m) &\le 2^p \sum_{m=1}^{\infty} m^{p-1} \alpha(\delta m)  \\
	&\le \frac{2^{2p-1}}{\delta^p} \sum_{m=1}^{\infty} m^{p-1} \alpha(m) < \infty
	\end{split}
	\end{equation}
	The second inequality can be shown using a similar construction as in \Cref{lem:finite_CDFsum}. Finally, the finiteness of the last summation follows from the assumed $p$-strongly mixing property.

\end{proof}

We have thus established the stochastic dominance property of order $p\ge0$ for those network edges that ensure that the network is SSC under the $p$-strongly mixing condition. 
As the next step, we show an elementary lemma associated with the AoI variables of a time-varying network. The lemma shows that the existence of stochastically dominating random variables associated with the AoI variables of a time-varying network is a transitive property.

\begin{lemma}
	\label{lem: key_lemma}
	For nodes $i,j,k \in \cV$ of a time-varying network suppose $\tau_{ij}(n)$ and $\tau_{jk}(n)$ are stochastically dominated by $\overline{\tau}_{ij}$ and $\overline{\tau}_{jk}$, respectively. Then
	\begin{enumerate}
		\item There is random variable $\overline{\tau}_{ik}$ that stochastically dominates $\tau_{ik}(n)$.
		\item If moreover $\Ew{\overline{\tau}_{ij}^p} + \Ew{\overline{\tau}_{jk}^p} < \infty$ for some $p>0$, then also $\Ew{\overline{\tau}_{ik}^p} < \infty$.
	\end{enumerate}
\end{lemma}
\begin{proof}
	Fix $i,j,k \in \cV$ and some $m\ge 2$. Now observe the following inclusion associated with events of the three AoI variables $\tau_{ij}(n), \tau_{jk}(n)$ and $\tau_{ik}(n)$:
	\begin{equation}
	\label{eq: keylemma_property}
	\{\tau_{ij}(n-\frac{m}{2}) \le \frac{m}{2}\} \cap \{\tau_{jk}(n) \le \frac{m}{2}\} \subset \{\tau_{ik}(n) \le m\}, 
	\end{equation}
	The inclusion states that the two events
	\begin{enumerate}
		\item The AoI is less than $\frac{m}{2}$ for information received at node $j$ from node $i$ at time $n-\frac{m}{2}$
		\item The AoI is less than $\frac{m}{2}$ for information received at node $k$ from node $j$ at time $n$
	\end{enumerate}
	imply the event that the AoI is less than $m$ for information received at node $k$ from node $i$ at time $n$.
	By taking the complement of the inclusion in \eqref{eq: keylemma_property}, we have that
	\begin{align}
	\Pr{\tau_{ik}(n) > m} &\le  \Pr{ \{\tau_{ij}(n-\frac{m}{2}) > \frac{m}{2}\} \cup \{\tau_{jk}(n) > \frac{m}{2}\} } \\
	&\le  \Pr{\tau_{ij}(n-\frac{m}{2}) > \frac{m}{2} }  + \Pr{\tau_{jk}(n) > \frac{m}{2}} \\
	&< \Pr{ \overline{\tau}_{ij} > \frac{m}{2}} + \Pr{\overline{\tau}_{jk} > \frac{m}{2}}.
	\end{align}
	In the last step, we used the assumption that there are random variables $\overline{\tau}_{ij}$ and $\overline{\tau}_{jk}$ that stochastically dominate $\tau_{ij}(n)$ and $\tau_{jk}(n)$, respectively, for all $n$. 
	
	Now $\overline{\tau}_{ij}$ and $\overline{\tau}_{jk}$ are integer-valued, so there is some $M \in \N$ such that 
	\begin{equation}
	\Pr{\overline{\tau}_{ij} > \frac{m}{2} }  + \Pr{\overline{\tau}_{jk} > \frac{m}{2}} < 1
	\end{equation} for all $m\ge M$. Define a non-negative integer-valued random variable $\overline{\tau}_{ik}$ by defining its CDF:
	\begin{align}
	\Pr{\overline{\tau}_{ik} > m} &\coloneqq 1, \quad \text{ for all } 0 \le m < M, \\
	\Pr{\overline{\tau}_{ik} > m} &\coloneqq \Pr{\overline{\tau}_{ij} > \frac{m}{2} }  + \Pr{\overline{\tau}_{jk} > \frac{m}{2}}, \quad \text{ otherwise}. 
	\end{align}
	This proves part $(a)$ of the lemma. 
	
	Now suppose $\Ew{\overline{\tau}_{ij}^p} + \Ew{\overline{\tau}_{jk}^p} < \infty$ for some $p>0$.
	We can now write the $p$-th moment of $\overline{\tau}_{ik}$ using its CDF from above:
	\begin{align}
	\Ew{\overline{\tau}_{ik}^p} &= \sum_{m=0}^\infty ((m+1)^{p}-m^{p}) \Pr{\overline{\tau}_{ik} > m} \\
	&\le \sum_{m=0}^\infty  ((m+1)^{p}-m^{p}) \Pr{\overline{\tau}_{ij} > \frac{m}{2} } + \sum_{m=0}^\infty  ((m+1)^{p}-m^{p}) \Pr{\overline{\tau}_{jk} > \frac{m}{2}} \\
	&=2^p\left(\Ew{\overline{\tau}_{ij}^p} + \Ew{\overline{\tau}_{jk}^p} \right) < \infty.
	\end{align}
	Where the equality follows from \Cref{eq:moment_eq}, since $2\overline{\tau}_{ij}$ and $2\overline{\tau}_{jk}$ are non-negative integer-valued random variables.
	This proves part $(b)$ of the lemma.
\end{proof}

\Cref{lem: key_lemma} allows that we extend the stochastic dominance properties from \Cref{lem: aoi} for node pairs $(i,j) \in \cE$ to arbitrary node pairs $(i,j) \in \cV^2$. We are now ready to prove \Cref{thm: 0}.

\begin{proof}[Proof of \Cref{thm: 0}]
	First, fix an arbitrary pairs of nodes $(i,j) \in \cV^2$. Since the network is SSC, it follows from \Cref{lem: aoi} there is a sequence of edges $\{(i_k,i_{k+1}) \}_{k=1}^{K-1} \in \cE$ for some $K \ge 1$, with $i_1 = i$ and $i_K = j$, such that for each $\tau_{i_ki_{k+1}}$, there is non-negative integer-valued random variable $\overline{\tau}_{i_ki_{k+1}}$ that stochastically dominates all $\tau_{i_ki_{k+1}}(n)$ for all $n \in \N_0$, with $\Ew{\overline{\tau}_{i_ki_{k+1}}^p} < \infty$. It now follows by induction using the transitive property of the AoI variables from \Cref{lem: key_lemma}(b), that there is a non-negative integer-valued random variable $\overline{\tau}_{ij}$ that stochastically dominates all $\tau_{ij}(n)$ for all $n \in \N_0$, with $\Ew{\overline{\tau}_{ij}^p} < \infty$.
	
	It is now left to verify that there is a single dominating random variables for all pairs $(i,j) \in \cV^2$. This essentially follows since we consider finitely many agents. For every $m \ge 0$, define
	\begin{equation}
	h(m) \coloneqq \sum_{(i,j) \in \cV^2} \Pr{\overline{\tau}_{ij} > m}.
	\end{equation}
	Since $|\cV^2| < \infty$, there is some $M \ge0$, such that $h(m) \le 1$ for all $m\ge M$. 
	Define a non-negative integer-valued random variable $\overline{\tau}$ by describing its CDF as follows:
	\begin{alignat}{2}
	\Pr{\overline{\tau}_{ij} > m} &= 1,  && 0\le  m < M, \\
	\Pr{\overline{\tau}_{ij} > m} &= h(m), \qquad && m\ge M. 
	\end{alignat}
	By construction $\overline{\tau}$ stochastically dominates all $\tau_{ij}(n)$ for all $(i,j) \in \cV^2$ and for all $n \in \N_0$. Finally, 
	we have 
	\begin{equation}
	\Ew{\overline{\tau}^p} \le \sum_{m=0}^\infty \left( (m+1)^{p}-m^{p} \right) h(m) = \sum_{(i,j) \in \cV^2} \Ew{\overline{\tau}_{ij}^p} < \infty, 
	\end{equation}
	where the equality simply follows from continuity of addition and since all $\Ew{\overline{\tau}_{ij}^p}$ are convergent.

\end{proof}

\section{Conclusions and future work}
\label{sec:conclusion}

In this work, we presented an asymptotic convergence analysis of distributed stochastic gradient descent that uses aged information. The required network assumptions have been weakened to the mere existence of non-negative integer-valued random variable with finite first moment that stochastically dominates all age of information random variables variables. This assumption can be satisfied with the new network \Cref{asm: network-A,asm: network-B}. These assumptions are significantly weaker then the common network assumptions in the literature. 
We hope that our assumptions penalize future work in distributed optimization under less restrictive network assumptions. Notably, instead of periodic or independent communication, we merely require asymptotically independent communication formulated using $\alpha$-mixing with the minimal requirement that $\sum_{n=0}^{\infty} \alpha(n) <\infty$.
 
It would be interesting to see, whether summability properties of $\alpha$-mixing coefficients indeed hold for representative physical wireless communication system. This might be possible when the underlying physical system has a mixing property in an ergodic sense. For example, hyperbolic systems are common models to describe electro magnetic wave propagation and it was shown in \cite{babillot2002mixing} that hyperbolic systems admit a strong mixing property in an ergodic sense.

To apply \Cref{asm: network-B} in practice, it would be most desirable if the $\alpha$-mixing coefficients (or an upper bound) for the network processes $\Ind{\bigcup_{k=n}^{n+\eta} A^n_{ij}}$ could be estimated from data. Unfortunately, there are only a handful of methods that estimate or approximate the mixing coefficients from data. One method that uses an approximation method based on histograms was presented in \cite{mcdonald2015estimating}. However, this method suffers from high complexity.
Very recently, a new method was presented in \cite{khaleghi2021inferring}. Most notably, the work presents a hypothesis test to decide, whether the sum of the alpha mixing coefficients is below an upper bound. With this, it is therefore now possible to verify with high confidence, whether \Cref{asm: network-B} holds for $p=1$ using data. 

\begin{funding}
Adrian Redder was supported by the German Research Foundation (DFG) - 315248657 and SFB 901.
\end{funding}


\bibliographystyle{imsart-nameyear.bst}
\bibliography{references}

\end{document}